\documentclass[oneside,british]{amsart}
\usepackage[T1]{fontenc}
\usepackage[utf8]{inputenc}
\usepackage{amsthm}
\usepackage{amssymb}
\usepackage{geometry}
\usepackage[authoryear]{natbib}

\makeatletter
\numberwithin{equation}{section}
\numberwithin{figure}{section}

\newcommand{\T}{\mathrm{\scriptscriptstyle T}}

\makeatother

\theoremstyle{plain}
\newtheorem{thm}{\protect\theoremname}
\theoremstyle{definition}
\newtheorem{defn}[thm]{\protect\definitionname}
\theoremstyle{plain}
\newtheorem{lem}[thm]{\protect\lemmaname}
\theoremstyle{remark}
\newtheorem{rem}[thm]{\protect\remarkname}
\theoremstyle{plain}
\newtheorem{prop}[thm]{\protect\propositionname}
\theoremstyle{definition}
\newtheorem{example}[thm]{\protect\examplename}
\usepackage{babel}
\providecommand{\definitionname}{Definition}
\providecommand{\examplename}{Example}
\providecommand{\lemmaname}{Lemma}
\providecommand{\propositionname}{Proposition}
\providecommand{\remarkname}{Remark}
\providecommand{\theoremname}{Theorem}

\begin{document}
\title{Robustness of random-walk Metropolis for steep potentials}
\author{Sam Power}
\address{School of Mathematics, University of Bristol}
\email{sam.power@bristol.ac.uk}
\begin{abstract}
In Markov chain Monte Carlo sampling, light-tailed target distributions
present something of a poisoned chalice: their light tails offer good
confinement, and tend to imply good mixing properties for natural
continuous-time dynamics, but the steepness of their tail decay means
that they often fall outside of the scope of modern quantitative convergence
theory. For usual gradient-based samplers, this reflects a genuine
instability issue, whereby Metropolis acceptance rates can degrade
badly. In this work, we study the gradient-free random-walk Metropolis
sampler, and show that for a wide range of light-tailed targets, the
acceptance probability remains stable for reasonable choices of proposal
variance, from which effective and favourable mixing time estimates
can be deduced. The analysis relies on a simple relationship between
the first and second derivatives of the log-density of the target
distribution.
\end{abstract}

\maketitle

\section{Introduction}

Random-walk Metropolis (RWM) is among the simplest and most widely
used Markov chain Monte Carlo algorithms. Given a target distribution
\[
\pi\left(\mathrm{d}x\right)\propto\exp\left(-U\left(x\right)\right)\,\mathrm{d}x,\qquad x\in\mathbf{R}^{d},
\]
the algorithm proposes to move from $x$ to $Y=x+\sigma\,Z$, for
a symmetric random increment $Z$, and accepts the proposal with probability
\[
\alpha\left(x,Y\right)=\min\left\{ 1,\exp\left(-U\left(Y\right)+U\left(x\right)\right)\right\} .
\]
Iterating this procedure then generates a Markov chain which is $\pi$-reversible
and ergodic under mild assumptions. This simplicity has made RWM a
standard benchmark in both theory and practice. 

Recent work by \citet{andrieu2024explicit,andrieu2026weak} has also
produced increasingly sharp nonasymptotic convergence guarantees.
A useful principle emerging from this literature is that much of the
mixing analysis can be separated into two parts: i) the geometry of
the target, handled through conductance, isoperimetry, or functional
inequalities, and ii) a local algorithmic estimate controlling the
acceptance probability. In particular, once the worst-case acceptance
probability is bounded away from zero, general arguments transfer
geometric information about the target into quantitative convergence
bounds for RWM.

Existing bounds on the worst-case acceptance probability typically
begin with a global modulus of continuity for the gradient of the
potential. One assumes, for example, that for some subadditive $\psi$,
one can bound
\[
\left\Vert \nabla U\left(x\right)-\nabla U\left(y\right)\right\Vert \leq\psi\left(\left\Vert x-y\right\Vert \right),
\]
uniformly over $x,y$. By integration, leads immediately to a state-independent
bound on the Taylor remainder $U\left(x+h\right)-U\left(x\right)-\langle\nabla U\left(x\right),h\rangle$,
and therefore makes it straightforward to select a proposal scale
$\sigma$ for which acceptance remains positive uniformly across the
state space; see Lemma 39 of \citet{andrieu2024explicit} for details.

While user-friendly and rather general, the price of these conventional
assumptions is nevertheless substantial in one key respect. A global,
homogeneous-in-space modulus of continuity $\psi$ prevents $\nabla U$
from growing faster than linearly and hence restricts $U$ to grow
no faster than quadratically. It therefore excludes elementary superquadratic
targets with good confinement and local smoothness such as 
\[
U\left(x\right)=\left(1+\left\Vert x\right\Vert ^{2}\right)^{p/2},\qquad p>2,
\]
for which one certainly expects the RWM to perform well.

For gradient-based algorithms, quadratic growth represents a meaningful
threshold. In the unadjusted Langevin algorithm and in numerical implementations
of Hamiltonian Monte Carlo, rapidly growing derivatives can cause
genuine numerical instability unless the step size is adapted to the
position, or otherwise constrained in more involved ways. Practical
approaches then either argue that this catastrophic non-smoothness
occurs only on a set of negligible mass under $\pi$, and can hence
be systematically ignored, or modify the algorithm directly to induce
stability; see, e.g., \citep{roberts1996exponential,atchade2006adaptive,bou2013nonasymptotic,brosse2019tamed,hodgkinson2021implicit,shukla2026proximal}
for details and instances of each.

It is much less clear that the same limitations should apply to RWM:
the proposal does not involve $\nabla U$, and proposals moving towards
regions of lower potential are accepted automatically. Far from the
centre of a superquadratic target, outward proposals may indeed have
very small acceptance probabilities, but their inward counterparts
typically have very large ones. Earlier qualitative results about
the long-time behaviour of RWM \citep{jarner2000geometric} indicate
that geometric ergodicity can indeed persist in some of these settings.
One therefore wonders whether low acceptance probabilities in unfavourable
directions obstruct uniform control of the average acceptance probability. 

The purpose of this article is to develop an acceptance analysis adapted
to this setting, working with assumptions which allow for much more
rapid growth of the derivatives of $U$, provided that they are self-bounding
in a suitable sense. Our working assumption will take the form of
a \emph{curvature-{}-force profile} estimate 
\[
\left\Vert \nabla^{2}U\left(x\right)\right\Vert _{\mathrm{op}}\leq G\left(\left\Vert \nabla U\left(x\right)\right\Vert \right),
\]
with $G$ some explicit nondecreasing function, typically chosen to
be concave, following work of \citet{ZhangHSJ20,li2023convex} in
the optimization community. This accounts for potentials $U$ of arbitrary
polynomial growth, exponential growth, and much in between. We will
see that under this assumption, acceptance rates can be controlled
rather reasonably.

The key argument then involves exhibiting that under such a smoothness
condition, for each proposal increment $z$, upon considering the
destinations $x\pm\sigma\,z$, while the associated energy increments
$U\left(x\pm\sigma\,z\right)-U\left(x\right)$ are capable of being
quite large indeed, it cannot be the case that \emph{both} proposals
are simultaneously strongly unfavourable, at least on the relevant
proposal scale. This gives a direct, nonasymptotic lower bound on
the average acceptance probability for the RWM proposal, uniformly
over the current state. In various worked examples, the resulting
proposal scales will be seen to depend inverse-polynomially on the
dimension and on the parameters of the curvature-force profile, leading
to overall complexity guarantees which scale polynomially with the
dimension. In the model case $U\left(x\right)\asymp\left\Vert x\right\Vert ^{p}$,
$p>2$, our results indicate stable behaviour upon taking $\sigma\asymp d^{-(p-1)/p}$,
which gracefully reduces to the usual $d^{-1/2}$ scaling in the quadratic
case. Superquadratic growth of the potential thus changes the appropriate
scale of the random walk, but it does not preclude uniform acceptance
for reasonably scaled proposals.

Combined with existing conductance and isoperimetric results, these
estimates substantially enlarge the class of targets for which quantitative
RWM convergence guarantees are available. The framework includes superquadratic
radial potentials, suitable nonradial and nonconvex perturbations,
and more general models whose curvature is controlled as a function
of their gradient. Importantly, convexity plays no essential role
in controlling the acceptance probabilities, which are governed by
quantities that are local and one-sided. Assumptions governing the
global geometry and mixing of the target can be treated separately,
and are in principle largely decoupled from these issues around acceptance
probabilities.

Beyond the resulting bounds for RWM, this analysis highlights a basic
distinction between zeroth-order and first-order algorithms. Rapid
growth of the gradient or Hessian can create an intrinsic stability
problem for discretization of dynamical systems which explicitly use
derivatives. By avoiding gradient information, the random-walk proposal
also avoids that instability. For RWM, the role of increasing steepness
is therefore largely to make one half of a symmetric proposal pair
easier to accept, with the other half then either favourable or irrelevant.
The quadratic-growth boundary encountered in earlier quantitative
RWM analyses based on globally Lipschitz gradients is therefore not
an intrinsic limitation of the algorithm, but rather results from
working only with a state-independent Taylor-remainder estimate. By
permitting the curvature to grow with the local force, and retaining
rather than discarding the favourable first-order contribution, one
obtains state-uniform acceptance bounds for a substantially broader
class of steep potentials.

\section{Preliminaries}\label{sec:preliminaries}

\subsection{Notation}

Throughout, $\left\Vert \cdot\right\Vert $ denotes the Euclidean
norm on $\mathbf{R}^{d}$ and the corresponding operator norm on matrices
when appropriate. We write $a_{+}:=\max\left\{ a,0\right\} $, $a\wedge b:=\min\left\{ a,b\right\} $,
$a\vee b:=\max\left\{ a,b\right\} $. The standard Gaussian distribution
on $\mathbf{R}^{d}$ is denoted by $\mathcal{N}\left(0,\mathbf{I}_{d}\right)$,
and $\Phi$ denotes the distribution function of a one-dimensional
standard Gaussian. All expectations and probabilities involving $Z$,
unless otherwise indicated, refer to $Z\sim\mathcal{N}\left(0,\mathbf{I}_{d}\right)$.

Let $\pi\left(\mathrm{d}x\right)=\mathcal{Z}^{-1}\exp\left(-U\left(x\right)\right)\,\mathrm{d}x$
be a probability distribution on $\mathbf{R}^{d}$, where $U\colon\mathbf{R}^{d}\to\mathbf{R}$
is twice continuously differentiable and $\mathcal{Z}=\int_{\mathbf{R}^{d}}\exp\left(-U\left(x\right)\right)\,\mathrm{d}x<\infty$.
Most of the arguments below only require that $\nabla U$ be locally
absolutely continuous along line segments and that the relevant Hessian
bounds hold almost everywhere; we work with a $C^{2}$ assumption
throughout to avoid inessential regularity issues.

\subsection{Random-walk Metropolis}

For a proposal scale $\sigma>0$, the Gaussian random-walk proposal
from $x\in\mathbf{R}^{d}$ is 
\[
Y=x+\sigma\,Z,\qquad Z\sim\mathcal{N}\left(0,\mathbf{I}_{d}\right).
\]
Define the energy increment $\Delta_{\sigma}\left(x,z\right):=U\left(x+\sigma z\right)-U\left(x\right)$
and the corresponding acceptance function 
\[
\alpha_{\sigma}\left(x,z\right):=1\wedge\exp\left(-\Delta_{\sigma}\left(x,z\right)\right)=\exp\left(-\left[\Delta_{\sigma}\left(x,z\right)\right]_{+}\right).
\]
The average acceptance probability at $x$ and its worst-case value
are, respectively,
\begin{align*}
\alpha_{\sigma}\left(x\right) & :=\mathbf{E}\left(\alpha_{\sigma}\left(x,Z\right)\right),\\
\underline{\alpha}_{\sigma} & :=\inf_{x\in\mathbf{R}^{d}}\alpha_{\sigma}\left(x\right).
\end{align*}

The results of \citet{andrieu2024explicit,andrieu2026weak} establish
the following general result about the mixing time of well-tuned random-walk
Metropolis kernels. We state the result somewhat informally and defer
the precise definitions to these original papers, on the basis that
they are inessential for the present developments.

\textbf{Informal principle 1.} Let $\pi$ be a probability measure
on $\mathbf{R}^{d}$, and let $\pi$ admit isoperimetric constant
$\kappa>0$. Let $P$ denote the random-walk Metropolis kernel with
proposal variance $\sigma^{2}$ and target $\pi$, and assume that
uniformly over $x\in\mathbf{R}^{d}$, the acceptance rate out of the
state $x$ is lower-bounded by a positive numerical constant $\alpha_{0}$.
Then $P$ admits a positive conductance $\mathsf{Cond}\left(P\right)$
and a positive spectral gap $\mathsf{Gap}\left(P\right)$ satisfying
\[
\mathsf{Cond}\left(P\right)\gtrsim\sigma\,\kappa,\qquad\mathsf{Gap}\left(P\right)\gtrsim\sigma^{2}\,\kappa^{2}.
\]
For additional details, consult Theorem 18 of \citet{andrieu2024explicit}.

\textbf{Informal principle 2.} Let $\pi$ be a probability measure
on $\mathbf{R}^{d}$, and let $\widetilde{I}_{\pi}$ be a regular
isoperimetric minorant for $\pi$. Let $P$ denote the random-walk
Metropolis kernel with proposal variance $\sigma^{2}$ and target
$\pi$, and assume that uniformly over $x\in\mathbf{R}^{d}$, the
acceptance rate out of the state $x$ is lower-bounded by a positive
numerical constant $\alpha_{0}$. Let $\mu$ be an initial distribution
for which $\chi^{2}\left(\mu,\pi\right)=:u_{0}<\infty$. Then, in
order to achieve $\chi^{2}\left(\mu P^{n},\pi\right)\in\mathcal{O}\left(1\right)$,
it is essentially sufficient to take 
\[
n\gtrsim\frac{1}{\sigma^{2}}\,\int^{1/4}_{2\,u^{-1}_{0}}\frac{\xi}{\widetilde{I}_{\pi}\left(\xi\right)^{2}}\,\mathrm{d}\xi,
\]
provided that the initial divergence $u_{0}$ is not impractically
large; see Section 5.1 of \citet{andrieu2024explicit} for the precise
statement.

The operational interpretation of these results is that once one is
able to tune the step-size $\sigma$ so that the worst-case acceptance
rate is lower-bounded by a constant such as $1/4$, it follows that
both the mixing time of the chain to an $\mathcal{O}\left(1\right)$
neighbourhood of the target distribution and the relaxation properties
close to equilibrium, as summarized by the conductance and spectral
gap, are characterized well by the step-size, and by algorithm-agnostic
features of the target distribution, i.e. the isoperimetric constant
and profile. On this basis, we focus the remainder of our discussion
precisely on establishing that this control on the acceptance rate
can be achieved, and defer the rather involved and more application-specific
isoperimetric calculations entirely.

Our analysis here focuses exclusively on the question of whether the
acceptance rate can be lower-bounded uniformly in $x$, which is conservative
by nature. In certain situations where typical smoothness properties
are much better than the worst case, more refined analyses based on
the so-called $s$-conductance can allow for sharper mixing time bounds;
we omit further discussion of this point.

\subsection{Curvature-force profiles}

We next formalize the smoothness condition used throughout the paper,
essentially as introduced by \citet{li2023convex}.
\begin{defn}
\label{def:curvature-growth} Let $G\colon\left[0,\infty\right)\to\left[0,\infty\right)$
be continuous and nondecreasing. We say that $U$ is $G$-smooth if,
for every $x\in\mathbf{R}^{d}$, 
\[
\left\Vert \nabla^{2}U\left(x\right)\right\Vert _{\mathrm{op}}\leq G\left(\left\Vert \nabla U\left(x\right)\right\Vert \right).
\]
We call $G$ a curvature-force profile for $U$.
\end{defn}

We will generally take $G$ to be concave; this ensures, in particular,
that $G$ has at most affine growth and that various functions introduced
in the sequel are plainly well-defined globally. Concavity per se
is convenient rather than fundamental; we will see that the arguments
apply to any nondecreasing profile for which all derived quantities
and functions are of sufficiently moderate growth.

A key special case is the following, studied by \citet{ZhangHSJ20}.
\begin{defn}
\label{def:L0L1} For $L_{0},L_{1}\geq0$, we say that $U$ is $\left(L_{0},L_{1}\right)$-smooth
if, for every $x\in\mathbf{R}^{d}$, 
\[
\left\Vert \nabla^{2}U\left(x\right)\right\Vert _{\mathrm{op}}\leq L_{0}+L_{1}\left\Vert \nabla U\left(x\right)\right\Vert .
\]
\end{defn}

Ordinary $L$-smoothness is the case $L_{0}=L$, $L_{1}=0$. In principle,
for a given $U$, one might be able to verify this condition for an
entire Pareto frontier of values $\Lambda\ni\left(L_{0},L_{1}\right)$;
in this case, minimizing the right-hand side of the above comparison
along this frontier then amounts to $G$-smoothness with $G_{\Lambda}\left(r\right):=\inf\left\{ L_{0}+L_{1}\,r:\left(L_{0},L_{1}\right)\in\Lambda\right\} $
concave by construction. Conversely, any suitable concave $G$ admits
an affine majorant of this form.

\subsection{The scalar comparison flow}

Let $G$ be a curvature-force profile. For $s,r\geq0$, let $\Gamma\left(r,s\right)$
denote the maximal nonnegative solution at time $r$ of the evolution
equation
\[
\frac{\partial}{\partial r}\Gamma\left(r,s\right)=G\left(\Gamma\left(r,s\right)\right),\qquad\Gamma\left(0,s\right)=s,
\]
i.e. the flow map for the ODE $\dot{\gamma}=G\left(\gamma\right)$,
initialized at $\gamma\left(0\right)=s$ and run for $r$ units of
time. 

When $G>0$ on $\left[s,\Gamma\left(r,s\right)\right]$, this is equivalently
characterized by the integral formula $\int^{\Gamma\left(r,s\right)}_{s}\mathrm{d}u/G\left(u\right)=r$.
The maximal-solution convention also covers profiles with $G\left(0\right)=0$,
such as the power profiles used in later examples. The flow map $\left(r,s\right)\mapsto\Gamma\left(r,s\right)$
is nondecreasing in both variables and satisfies the semigroup identity
$\Gamma\left(r+t,s\right)=\Gamma\left(r,\Gamma\left(t,s\right)\right)$. 

The reason for introducing this flow is the following elementary comparison
estimate.
\begin{lem}
\label{lem:gradient-comparison} Suppose that $U$ is $G$-smooth.
Then, for every $x,h\in\mathbf{R}^{d}$, it holds that 
\[
\left\Vert \nabla U\left(x+h\right)\right\Vert \leq\Gamma\left(\left\Vert h\right\Vert ,\left\Vert \nabla U\left(x\right)\right\Vert \right).
\]
\end{lem}

\begin{proof}
Set $q\left(t\right)=\left\Vert \nabla U\left(x+t\,h\right)\right\Vert $.
At every point at which $q$ is differentiable, $G$-smoothness supplies
the bound
\[
q'\left(t\right)\leq\left\Vert \nabla^{2}U\left(x+t\,h\right)h\right\Vert \leq\left\Vert h\right\Vert \,G\left(q\left(t\right)\right).
\]
The claim then follows from general comparison principles for scalar
ODEs. 
\end{proof}

To retain the averaging present in the exact Taylor formula, define
\[
\overline{G}\left(r,s\right):=2\,\int^{1}_{0}\left(1-t\right)\,G\left(\Gamma\left(t\,r,s\right)\right)\,\mathrm{d}t.
\]
We refer to $\overline{G}$ as the \emph{averaged} curvature profile.
It satisfies $G\left(s\right)\leq\overline{G}\left(r,s\right)\leq G\left(\Gamma\left(r,s\right)\right)$,
but can be substantially smaller than the upper endpoint $G\left(\Gamma\left(r,s\right)\right)$.
\begin{lem}
\label{lem:taylor-remainder} If $U$ is $G$-smooth, then 
\[
\left|U\left(x+h\right)-U\left(x\right)-\left\langle \nabla U\left(x\right),h\right\rangle \right|\leq\frac{\left\Vert h\right\Vert ^{2}}{2}\,\overline{G}\left(\left\Vert h\right\Vert ,\left\Vert \nabla U\left(x\right)\right\Vert \right).
\]
\end{lem}

\begin{proof}
The integral form of Taylor's theorem gives that
\[
U\left(x+h\right)-U\left(x\right)-\left\langle \nabla U\left(x\right),h\right\rangle =\int^{1}_{0}\left(1-t\right)\,h^{\T}\nabla^{2}U\left(x+t\,h\right)h\,\mathrm{d}t.
\]
Apply then Definition~\ref{def:curvature-growth} and Lemma~\ref{lem:gradient-comparison}. 
\end{proof}

\begin{rem}
In principle, we could commence our analysis directly from a control
of this form: many of the developments which follow will require only
the force-adaptive Taylor bound
\[
U\left(x+h\right)-U\left(x\right)-\left\langle \nabla U\left(x\right),h\right\rangle \leq\frac{\left\Vert h\right\Vert ^{2}}{2}\,\mathsf{L}\left(\left\Vert h\right\Vert ,\left\Vert \nabla U\left(x\right)\right\Vert \right),
\]
for some suitable local smoothness function $\mathsf{L}$, obtained
by whatever means are appropriate. Nevertheless, we formulate our
assumptions in terms of curvature-{}-force profiles because they are
relatively easy to verify and connect explicitly to the prior literature.
For more exotic potentials $U$, the sharpest analyses may instead
benefit from a bespoke expansion of the displayed form.
\end{rem}

For the affine profile $G\left(s\right)=L_{0}+L_{1}\,s$, the preceding
quantities are explicit. If $L_{1}>0$, then one computes that
\begin{align*}
\Gamma\left(r,s\right) & =\left(s+\frac{L_{0}}{L_{1}}\right)\,\exp\left(L_{1}\,r\right)-\frac{L_{0}}{L_{1}}\\
G\left(\Gamma\left(r,s\right)\right) & =\exp\left(L_{1}\,r\right)\,\left(L_{0}+L_{1}\,s\right)\\
\overline{G}\left(r,s\right) & =2\,\varphi\left(L_{1}\,r\right)\,\left(L_{0}+L_{1}\,s\right),
\end{align*}
where $\varphi\left(u\right):=u^{-2}\,\left\{ \exp\left(u\right)-1-u\right\} $
for $u>0$. One checks that for small $u$, $2\,\varphi\left(u\right)=1+u/3+\mathcal{O}\left(u^{2}\right)$
and $\exp\left(u\right)=1+u+\mathcal{O}\left(u^{2}\right)$, i.e.
$\overline{G}\left(r,s\right)\leq G\left(\Gamma\left(r,s\right)\right)$
as anticipated.

\subsection{Antipodal proposals}

The central observation is that, for a symmetric random-walk proposal,
it is enough that one of the two increments $\pm h$ be favourable.
Define the preferred energy increment associated with $z$ at $x$
by
\[
\underbar{\ensuremath{\Delta}}_{\sigma}\left(x,z\right):=\min\left\{ U\left(x+\sigma\,z\right)-U\left(x\right),U\left(x-\sigma\,z\right)-U\left(x\right)\right\} .
\]

\begin{lem}
\label{lem:antipodal-energy} Suppose that $U$ is $G$-smooth. Then
\[
\underbar{\ensuremath{\Delta}}_{\sigma}\left(x,z\right)\leq-\sigma\,\left|\langle\nabla U\left(x\right),z\rangle\right|+\frac{\sigma^{2}\,\left\Vert z\right\Vert ^{2}}{2}\,\overline{G}\left(\sigma\,\left\Vert z\right\Vert ,\left\Vert \nabla U\left(x\right)\right\Vert \right).
\]
\end{lem}

\begin{proof}
Apply Lemma~\ref{lem:taylor-remainder} separately to $h=\pm\sigma\,z$
, obtaining
\[
U\left(x\pm\sigma\,z\right)-U\left(x\right)\leq\pm\langle\nabla U\left(x\right),\sigma\,z\rangle+\frac{\sigma^{2}\,\left\Vert z\right\Vert ^{2}}{2}\,\overline{G}\left(\sigma\,\left\Vert z\right\Vert ,\left\Vert \nabla U\left(x\right)\right\Vert \right),
\]
and take the minimum of the two bounds to prove the claim. 
\end{proof}

The corresponding acceptance estimate uses only symmetry.
\begin{lem}
\label{lem:antipodal-acceptance} Let $Z$ have a centrally symmetric
distribution. Then 
\[
\alpha_{\sigma}\left(x\right)\geq1/2\,\mathbf{E}\left(\exp\left(-\left[\underbar{\ensuremath{\Delta}}_{\sigma}\left(x,Z\right)\right]_{+}\right)\right).
\]
Consequently, if $E$ is an event invariant under $Z\mapsto-Z$ and
$\underbar{\ensuremath{\Delta}}_{\sigma}\left(x,z\right)\leq b$ on
$E$, then 
\[
\alpha_{\sigma}\left(x\right)\geq1/2\,\exp\left(-b_{+}\right)\,\mathbf{P}\left(E\right).
\]
\end{lem}

\begin{proof}
By symmetry, 
\begin{align*}
\alpha_{\sigma}\left(x\right) & =1/2\,\left\{ \mathbf{E}\left(\exp\left(-\left[\Delta_{\sigma}\left(x,Z\right)\right]_{+}\right)+\exp\left(-\left[\Delta_{\sigma}\left(x,-Z\right)\right]_{+}\right)\right)\right\} \\
 & \geq1/2\,\mathbf{E}\left(\exp\left(-\left[\underbar{\ensuremath{\Delta}}_{\sigma}\left(x,Z\right)\right]_{+}\right)\right),
\end{align*}
and the first claim follows. The second follows by restricting the
expectation to $E$. 
\end{proof}

This makes one-sided estimates straightforward to apply. Usual bounds
on the absolute Taylor remainder would attempt to control both $U\left(x+h\right)-U\left(x\right)$
and $U\left(x-h\right)-U\left(x\right)$, leaving the analyst at the
mercy of the worst of the two increments. Lemma~\ref{lem:antipodal-energy}
instead controls their minimum, retaining the favourable linear contribution
$-\left|\left\langle \nabla U\left(x\right),h\right\rangle \right|$.

\subsection{A uniform acceptance bound}\label{subsec:tunable-acceptance}

The preceding estimates can be localized as follows. For $r>0$, define
\[
\mathcal{K}_{r,\sigma}\left(G\right):=\frac{r^{2}}{2}\,\sup\left\{ \frac{\overline{G}\left(r,s\right)}{1+\sigma\,s}:s\geq0\right\} .
\]
If $G$ is nonnegative, nondecreasing and concave, then one can verify
directly that this quantity is finite: concavity of $G$ implies a
linear envelope of the form $G\left(s\right)\leq A+B\,s=:G_{\mathrm{lin}}\left(s\right)$,
ODE comparison gives that $\Gamma\left(r,s\right)\leq\Gamma_{\mathrm{lin}}\left(r,s\right)$,
monotonicity then grants that
\[
G\left(\Gamma\left(r,s\right)\right)\leq G_{\mathrm{lin}}\left(\Gamma_{\mathrm{lin}}\left(r,s\right)\right)=\exp\left(B\,r\right)\,\left(A+B\,s\right),
\]
and one thus extracts the bounds
\begin{align*}
\sup\left\{ \frac{\overline{G}\left(r,s\right)}{1+\sigma\,s}:s\geq0\right\}  & \leq\exp\left(B\,r\right)\,\max\left\{ A,\frac{B}{\sigma}\right\} \\
\implies\qquad\mathcal{K}_{r,\sigma}\left(G\right) & \leq\exp\left(B\,r\right)\,\max\left\{ A,\frac{B}{\sigma}\right\} \,\frac{r^{2}}{2}.
\end{align*}
Using this definition, one sees that for $\left\Vert h\right\Vert \leq r$,
it holds that (writing $K=\mathcal{K}_{r,\sigma}\left(G\right)$)
\[
\frac{\left\Vert h\right\Vert ^{2}}{2}\,\overline{G}\left(\left\Vert h\right\Vert ,\left\Vert \nabla U\left(x\right)\right\Vert \right)\leq K\,\left(1+\sigma\,\left\Vert \nabla U\left(x\right)\right\Vert \right),
\]
and hence that when $\left\Vert \sigma\,z\right\Vert \leq r$, the
preferred energy increment is controlled as
\begin{align*}
\underbar{\ensuremath{\Delta}}_{\sigma}\left(x,z\right) & \leq-\sigma\,\left|\left\langle \nabla U\left(x\right),z\right\rangle \right|+K\,\left(1+\sigma\,\left\Vert \nabla U\left(x\right)\right\Vert \right)\\
 & =K+\sigma\,\left\Vert \nabla U\left(x\right)\right\Vert \,\left\{ K-\left|\left\langle \frac{\nabla U\left(x\right)}{\left\Vert \nabla U\left(x\right)\right\Vert },z\right\rangle \right|\right\} ,
\end{align*}
with the obvious interpretation for the latter expression when $\nabla U\left(x\right)=0$.
In effect, finiteness of $K$ at a given scale supplies a non-trivial
ball of position increments $z$ for which the associated energy increments
are not too unfavourable.

For $k\geq1$, let $p_{k}\left(c\right):=\mathbf{P}\left(\chi^{2}_{k}\leq c\,k\right)$,
$p_{0}\left(c\right):=1$. The following proposition separates the
curvature estimate from the Gaussian concentration estimate.
\begin{prop}
\label{prop:tunable-acceptance} Suppose that $U$ is $G$-smooth.
Fix $c>1$ and set $K=\mathcal{K}_{r,\sigma}\left(G\right)$ for $r=\sigma\,\left(c\,d\right)^{1/2}$.
If $K<c^{1/2}$, then
\[
\underline{\alpha}_{\sigma}\geq\exp\left(-K\right)\,p_{d-1}\left(c\right)\,\left\{ \Phi\left(c^{1/2}\right)-\Phi\left(K\right)\right\} .
\]
\end{prop}

\begin{proof}
Decompose $Z=W\,\nabla U\left(x\right)/\left\Vert \nabla U\left(x\right)\right\Vert +Z_{\perp}$
with $W\sim\mathcal{N}\left(0,1\right)$, $\left\Vert Z_{\perp}\right\Vert ^{2}\sim\chi^{2}_{d-1}$,
and the two random variables independent, and consider the event $\mathcal{A}_{c,K}:=\left\{ K\leq\left|W\right|\leq c^{1/2},\quad\left\Vert Z_{\perp}\right\Vert ^{2}\leq c\,\left(d-1\right)\right\} $.
On this event, $\left\Vert Z\right\Vert ^{2}\leq c\,d$ and $\left|\left\langle \nabla U\left(x\right)/\left\Vert \nabla U\left(x\right)\right\Vert ,Z\right\rangle \right|\geq K$,
whereby $\underbar{\ensuremath{\Delta}}_{\sigma}\left(x,Z\right)\leq K$
(with appropriate adaptation when $\nabla U\left(x\right)=0$). The
event is invariant under $Z\mapsto-Z$, and so we are free to apply
Lemma \ref{lem:antipodal-acceptance} and deduce that $\underline{\alpha}_{\sigma}\geq1/2\,\exp\left(-K\right)\,\mathbf{P}\left(\mathcal{A}_{c,K}\right)$.
Independence of $W$ and $\left\Vert Z_{\perp}\right\Vert ^{2}$ then
allows the explicit calculation $\mathbf{P}\left(\mathcal{A}_{c,K}\right)=2\,p_{d-1}\left(c\right)\,\left\{ \Phi\left(c^{1/2}\right)-\Phi\left(K\right)\right\} $,
and we conclude.
\end{proof}

The chi-square factor in Proposition~\ref{prop:tunable-acceptance}
admits standard explicit estimates. A Chernoff bound gives that for
every $k\geq1$ and $c>1$, $1-p_{k}\left(c\right)\leq\exp\left(-k\,I\left(c\right)\right)$,
with $I\left(c\right)=1/2\,\left(c-1-\log c\right)$, i.e. $p_{d-1}\left(c\right)$
is exponentially close to $1$ as $d$ grows. A convenient estimate
which tracks the impact of dimension more explicitly follows from
the Laurent--Massart inequality of \citet{laurent2000adaptive},
which states that
\[
\mathbf{P}\!\left\{ \chi^{2}_{k}>k+2\,\left(k\,u\right)^{1/2}+2\,u\right\} \leq\exp\left(-u\right),\qquad u>0.
\]
For $d\geq2$ and $\delta\in(0,1)$, define $c_{d,\delta}:=1+2\,\left\{ \log\left(1/\delta\right)/(d-1)\right\} ^{1/2}+2\log\left(1/\delta\right)/(d-1)$,
so that $p_{d-1}\left(c_{d,\delta}\right)\geq1-\delta$, and hence
\[
\underline{\alpha}_{\sigma}\geq\left(1-\delta\right)\,\exp\left(-K\right)\,\left\{ \Phi\left(c^{1/2}_{d,\delta}\right)-\Phi\left(K\right)\right\} ,\qquad K=\mathcal{K}^{\left(c_{d,\delta}\right)}_{d,\sigma}\left(G\right),
\]
provided that $K<c^{1/2}_{d,\delta}$.

For example, if $\sigma$ is chosen sufficiently small that $K\leq1/4$,
then, since $c_{d,\delta}>1$, one can bound
\[
\underline{\alpha}_{\sigma}\geq\left(1-\delta\right)\,\exp\left(-1/4\right)\,\left(\Phi\left(1\right)-\Phi\left(1/4\right)\right)>1/6\,\left(1-\delta\right).
\]
The simple criterion $\mathcal{K}^{\left(c_{d,\delta}\right)}_{d,\sigma}\left(G\right)\leq1/4$
thus ensures a dimension-uniform, state-uniform lower bound on the
average acceptance probability.

\section{Examples}\label{sec:examples}

\subsection{General setup}

We illustrate the curvature-force framework on several representative
classes of targets. Throughout this section, fix $c>1$, and write
$\mathcal{K}^{\left(c\right)}_{d,\sigma}\left(G\right)$ for the value
of $\mathcal{K}_{r,\sigma}\left(G\right)$ when $r=\sigma\,\left(c\,d\right)^{1/2}$.
Proposition~\ref{prop:tunable-acceptance} shows that it is enough
to take $\sigma$ sufficiently small that $\mathcal{K}^{\left(c\right)}_{d,\sigma}\left(G\right)$
is smaller than a numerical constant; we use the convenient threshold
$\mathcal{K}^{\left(c\right)}_{d,\sigma}\left(G\right)\leq1/4$ throughout.
For every fixed $c>1$, Proposition~\ref{prop:tunable-acceptance}
and the chi-square Chernoff bound then imply that the worst-case acceptance
probability is bounded away from zero uniformly in $d$. Sharper dimension-dependent
constants can be obtained directly from Proposition~\ref{prop:tunable-acceptance}.

\subsection{Potentials with bounded curvature}

Suppose first that $\left\Vert \nabla^{2}U\left(x\right)\right\Vert _{\mathrm{op}}\leq L$
globally. This corresponds to the constant profile $G\left(s\right)=L$.
Since $\overline{G}\left(r,s\right)=L$, we have the exact expression
$\mathcal{K}^{\left(c\right)}_{d,\sigma}\left(G\right)=\left(c/2\right)\,L\,d\,\sigma^{2}$.
Consequently, $\sigma\leq\left(2\,c\,L\,d\right)^{-1/2}$ implies
that $\mathcal{K}^{\left(c\right)}_{d,\sigma}\left(G\right)\leq1/4$,
and hence gives a uniform lower bound on the acceptance probability.
This recovers the familiar proposal scale $\sigma\asymp\left(L\,d\right)^{-1/2}$,
as advocated in \citet{andrieu2024explicit,andrieu2026weak}. 

\subsection{Power-law curvature profile}

We next consider profiles of the form $G\left(s\right)=L_{0}+L_{\beta}\,s^{\beta}$,
$0<\beta<1$, with $L_{0}\geq0$ and $L_{\beta}>0$. Before turning
to particular potentials, we record a convenient estimate for the
corresponding scalar flow.

Set $\gamma=\beta/(1-\beta)$, $D_{\beta}=2^{1-\beta}\,L_{\beta}$,
$C_{\gamma}=2^{\left(\gamma-1\right)_{+}}$, and $m_{\beta}:=\sup\left\{ u^{\beta}/(1+u):u\geq0\right\} =\beta^{\beta}\,\left(1-\beta\right)^{1-\beta}$.
A specific application of Jensen's inequality gives that
\[
L_{0}+L_{\beta}\,y^{\beta}\leq D_{\beta}\,\left\{ \left(L_{0}/L_{\beta}\right)^{1/\beta}+y\right\} ^{\beta},
\]
and solution of the resulting scalar differential inequality gives
that
\[
G\left(\Gamma\left(r,s\right)\right)\leq D_{\beta}\,C_{\gamma}\,\left(\frac{L_{0}}{L_{\beta}}+s^{\beta}+\left(\left(1-\beta\right)\,D_{\beta}\,r\right)^{\gamma}\right).
\]
Since $\overline{G}\left(r,s\right)\leq G\left(\Gamma\left(r,s\right)\right)$,
it follows that 
\[
\begin{aligned}\mathcal{K}^{\left(c\right)}_{d,\sigma}\left(G\right)\leq & \frac{c\,d\,\sigma^{2}}{2}\,D_{\beta}\,C_{\gamma}\,\left\{ \frac{L_{0}}{L_{\beta}}+m_{\beta}\,\sigma^{-\beta}+\left[\left(1-\beta\right)\,D_{\beta}\,\sigma\,\left(c\,d\right)^{1/2}\right]^{\gamma}\right\} .\end{aligned}
\]
In particular, there is an explicit constant $C_{\beta,c}<\infty$,
depending only on $\beta$ and $c$, such that 
\[
\mathcal{K}^{\left(c\right)}_{d,\sigma}\left(G\right)\leq C_{\beta,c}\,\left(L_{0}\,d\,\sigma^{2}+L_{\beta}\,d\,\sigma^{2-\beta}+\left(L_{\beta}\,d\,\sigma^{2-\beta}\right)^{1/(1-\beta)}\right).
\]
Thus, for a sufficiently small constant $\eta=\eta\left(\beta,c\right)$,
the choice $\sigma=\eta\,\min\left\{ \left(L_{0}\,d\right)^{-1/2},\left(L_{\beta}\,d\right)^{-1/(2-\beta)}\right\} $
ensures that $\mathcal{K}^{\left(c\right)}_{d,\sigma}\left(G\right)\leq1/4$.
\begin{example}
Consider $U_{p}\left(x\right)=p^{-1}\left\Vert x\right\Vert ^{p}$,
$p>2$. For $x\neq0$,
\[
\nabla U_{p}\left(x\right)=\left\Vert x\right\Vert ^{p-2}x,\qquad\nabla^{2}U_{p}\left(x\right)=\left\Vert x\right\Vert ^{p-2}\mathbf{I}_{d}+\left(p-2\right)\,\left\Vert x\right\Vert ^{p-4}xx^{\T}.
\]
 It follows that $\left\Vert \nabla U_{p}\left(x\right)\right\Vert =\left\Vert x\right\Vert ^{p-1}$,
$\left\Vert \nabla^{2}U_{p}\left(x\right)\right\Vert _{\mathrm{op}}=\left(p-1\right)\,\left\Vert x\right\Vert ^{p-2}$,
and therefore 
\[
\left\Vert \nabla^{2}U_{p}\left(x\right)\right\Vert _{\mathrm{op}}=\left(p-1\right)\,\left\Vert \nabla U_{p}\left(x\right)\right\Vert ^{(p-2)/(p-1)},
\]
 i.e. $U$ has polynomial curvature-force profile with exponent $\beta=(p-2)/(p-1)$,
for which $1/(2-\beta)=(p-1)/p$. Thus, for every sufficiently small
$\eta_{p,c}>0$, the proposal scale $\sigma=\eta_{p,c}\,\left(\left(p-1\right)\,d\right)^{-(p-1)/p}$
satisfies the threshold above. In particular, the worst-case acceptance
probability is bounded away from zero uniformly in the state and dimension.
\end{example}

In this example, the dimensional scaling is sharp. To see this, let
$x_{d}=d^{1/p}e_{1}$, $\tau_{d}:=\sigma_{d}d^{(p-1)/p}$, and decompose
$Z=We_{1}+Z_{\perp}$, where $W\sim\mathcal{N}\left(0,1\right)$.
Then 
\[
\begin{aligned}U_{p}\left(x_{d}+\sigma_{d}Z\right)-U_{p}\left(x_{d}\right) & =\frac{d}{p}\left[\left(1+\frac{2\tau_{d}W+\tau^{2}_{d}\left\Vert Z\right\Vert ^{2}/d}{d}\right)^{p/2}-1\right].\end{aligned}
\]
Consequently, if $\tau_{d}\to\tau\in\left[0,\infty\right)$ as $d\to\infty$,
the law of large numbers and a first-order expansion give that $U_{p}\left(x_{d}+\sigma_{d}Z\right)-U_{p}\left(x_{d}\right)\to\tau W+\tau^{2}/2$;
since the Metropolis acceptance function is bounded and continuous,
it follows that $\alpha_{\sigma_{d}}\left(x_{d}\right)\longrightarrow\mathbf{E}\left(1\wedge\exp\left(-\tau W-\tau^{2}/2\right)\right)=2\Phi\left(-\tau/2\right)$.
It thereby follows that scale $\sigma_{d}\asymp d^{-(p-1)/p}$ produces
a non-vanishing acceptance probability, as established.

Conversely, the same limiting argument gives a route towards showing
that this scale cannot be improved while retaining a state-uniform
acceptance bound. Indeed, convexity of $s\mapsto s^{p/2}$ gives the
pointwise estimate $U_{p}\left(x_{d}+\sigma_{d}Z\right)-U_{p}\left(x_{d}\right)\geq\tau_{d}W+\tau^{2}_{d}\left\Vert Z\right\Vert ^{2}/(2d)$,
and so on the event that $\left\{ W\geq-\tau_{d}/8,\quad\left\Vert Z\right\Vert ^{2}\geq d/2\right\} $,
the energy increment is therefore at least $1/8\,\tau^{2}_{d}$. It
follows, for every $d$, 
\[
\alpha_{\sigma_{d}}\left(x_{d}\right)\leq\Phi\left(-\frac{\tau_{d}}{8}\right)+\mathbf{P}\left(\chi^{2}_{d}<\frac{d}{2}\right)+\exp\left(-\frac{\tau^{2}_{d}}{8}\right),
\]
and so if $\tau_{d}\to\infty$ then $\alpha_{\sigma_{d}}\left(x_{d}\right)\to0$,
and hence $\underline{\alpha}_{\sigma_{d}}\leq\alpha_{\sigma_{d}}\left(x_{d}\right)\to0$
as well. The obstruction occurs on the typical radial scale of the
target: under $\pi$, one has $d^{-1}\,\left\Vert X\right\Vert ^{p}\to1$
in probability.

The same curvature-{}-force exponent also appears for non-radial product
targets.
\begin{example}
Consider the separable potential $U^{\mathrm{sep}}_{p}\left(x\right)=p^{-1}\sum^{d}_{i=1}\left|x_{i}\right|^{p}$.
Explicit calculations give that $\left\Vert \nabla^{2}U^{\mathrm{sep}}_{p}\left(x\right)\right\Vert _{\mathrm{op}}=\left(p-1\right)\,\max_{i}\left|x_{i}\right|^{p-2}$,
$\left\Vert \nabla U^{\mathrm{sep}}_{p}\left(x\right)\right\Vert \geq\max_{i}\left|x_{i}\right|^{p-1}$,
and hence 
\[
\left\Vert \nabla^{2}U^{\mathrm{sep}}_{p}\left(x\right)\right\Vert _{\mathrm{op}}\leq\left(p-1\right)\,\left\Vert \nabla U^{\mathrm{sep}}_{p}\left(x\right)\right\Vert ^{(p-2)/(p-1)}
\]
 once more.
\end{example}

In this example, however, the proposal scale prescribed by the general
curvature-{}-force argument is conservative. Classical optimal-scaling
theory for product targets suggests instead taking $\sigma=\ell\,d^{-1/2}$;
see \citep{roberts1997}. This scaling also yields a nonasymptotic,
state-uniform acceptance bound. Indeed, a coordinatewise Taylor estimate
gives that for some explicit $\left(A_{p},B_{p}\right)$, there holds
the bound
\[
\begin{aligned}\underline{\Delta}_{\sigma}\left(x,Z\right) & \leq-\sigma\,\left|\left\langle \nabla U^{\mathrm{sep}}_{p}\left(x\right),Z\right\rangle \right|+R_{p}\left(x,\sigma\,Z\right),\\
R_{p}\left(x,\sigma\,Z\right) & :=A_{p}\,\sigma^{2}\,\sum^{d}_{i=1}\left|x_{i}\right|^{p-2}\,Z^{2}_{i}+B_{p}\,\sigma^{p}\,\sum^{d}_{i=1}\left|Z_{i}\right|^{p}\geq0.
\end{aligned}
\]
Scaling $\sigma=\ell\,d^{-1/2}$, one can consider the event on which
$\left|\left\langle \nabla U^{\mathrm{sep}}_{p}\left(x\right),Z\right\rangle \right|\geq1/2\,\left\Vert \nabla U^{\mathrm{sep}}_{p}\left(x\right)\right\Vert $
and $R_{p}\left(x,\sigma\,Z\right)\leq4\,\mathbf{E}\left(R_{p}\left(x,\sigma\,Z\right)\right)$,
which is both reasonably likely and supplies the control $\underline{\Delta}_{\sigma}\left(x,Z\right)\leq C_{p}\,\ell^{p}$
for some explicit $C_{p}$. Lemma~\ref{lem:antipodal-acceptance}
therefore implies a lower bound of the form $\alpha_{\ell\,d^{-1/2}}\left(x\right)\geq c\,\exp\left(-C_{p}\,\ell^{p}\right)$
for some explicit $c>0$.

In this example, the usual product-target scale successfully controls
the worst-case acceptance probability. The general curvature-{}-force
estimate loses this improvement because its operator-norm bound discards
the effect of averaging across coordinates. Nevertheless, the relatively
crude bound still supplies an estimate with polynomial dependence
on the dimension.

The pure power can also be regularized near the origin without changing
the dimensional scaling.
\begin{example}
The example mentioned in the introduction, $\widetilde{U}_{p}\left(x\right)=\left(1+\left\Vert x\right\Vert ^{2}\right)^{p/2}$,
$p>2$, has 
\begin{align*}
\nabla\widetilde{U}_{p}\left(x\right) & =p\,\left(1+\left\Vert x\right\Vert ^{2}\right)^{(p-2)/2}x.\\
\nabla^{2}\widetilde{U}_{p}\left(x\right) & =p\,\left(1+\left\Vert x\right\Vert ^{2}\right)^{(p-2)/2}\mathbf{I}_{d}+p\,\left(p-2\right)\,\left(1+\left\Vert x\right\Vert ^{2}\right)^{(p-4)/2}xx^{\T}.
\end{align*}
Writing again $\beta=(p-2)/(p-1)$, one obtains the global bound 
\[
\left\Vert \nabla^{2}\widetilde{U}_{p}\left(x\right)\right\Vert _{\mathrm{op}}\leq L_{0,p}+L_{\beta,p}\left\Vert \nabla\widetilde{U}_{p}\left(x\right)\right\Vert ^{\beta},
\]
where one can take, for example, $L_{0,p}=p\,\left(p-1\right)\,2^{(p-2)/2}$,
$L_{\beta,p}=\left(p-1\right)\,p^{1-\beta}\,2^{\beta/2}$. It follows
that for sufficiently small $\eta_{p,c}>0$, the step-size $\sigma=\eta_{p,c}\,\min\left\{ \left(L_{0,p}\,d\right)^{-1/2},\left(L_{\beta,p}\,d\right)^{-(p-1)/p}\right\} $
suffices for uniform acceptance. For fixed $p$, this is of order
$d^{-(p-1)/p}$ as $d\to\infty$. In contrast, treating this potential
only through an affine $\left(L_{0},L_{1}\right)$ bound would produce
the more conservative scale $d^{-1}$.
\end{example}

\subsection{Affine curvature growth and an exponential example}

Suppose that $U$ is $\left(L_{0},L_{1}\right)$-smooth, i.e. 
\[
\left\Vert \nabla^{2}U\left(x\right)\right\Vert _{\mathrm{op}}\leq L_{0}+L_{1}\,\left\Vert \nabla U\left(x\right)\right\Vert .
\]
 From the explicit calculation in Section~\ref{sec:preliminaries},
we obtain that
\[
\mathcal{K}^{\left(c\right)}_{d,\sigma}=c\,\varphi\left(L_{1}\,\sigma\,\left(c\,d\right)^{1/2}\right)\,\max\left\{ L_{0}\,d\,\sigma^{2},\,L_{1}\,d\,\sigma\right\} .
\]
Consequently, a sufficient proposal scale is 
\[
\sigma=\eta\,\min\left\{ \left(L_{0}\,d\right)^{-1/2},\left(L_{1}\,d\right)^{-1}\right\} ,
\]
where $\eta>0$ is a sufficiently small numerical constant depending
only on $c$. Indeed, for $\eta\leq1$, $L_{0}\,d\,\sigma^{2}\leq\eta^{2}$,
$L_{1}\,d\,\sigma\leq\eta$, and $L_{1}\,\sigma\,\left(c\,d\right)^{1/2}\leq\eta\,\left(c/d\right)^{1/2}\leq\eta\,c^{1/2}$,
whereby $\mathcal{K}^{\left(c\right)}_{d,\sigma}\leq c\,\varphi\left(\eta\,c^{1/2}\right)\,\eta$.
Since $\varphi$ is increasing on $\left[0,\infty\right)$ and $\varphi\left(1\right)=\exp\left(1\right)-2\in\left[1/2,3/4\right]$,
the explicit choice
\[
0<\eta\leq\eta_{c}:=\min\left\{ c^{-1/2},\frac{1}{4\,c\,\left(\exp\left(1\right)-2\right)}\right\} ,
\]
guarantees that $\mathcal{K}^{\left(c\right)}_{d,\sigma}\leq1/4$. 

As a concrete example, consider $U\left(x\right)=\cosh\left(\left\Vert x\right\Vert \right)$,
which satisfies that
\[
\left\Vert \nabla^{2}U\left(x\right)\right\Vert _{\mathrm{op}}=\left\{ 1+\left\Vert \nabla U\left(x\right)\right\Vert ^{2}\right\} ^{1/2}\leq1+\left\Vert \nabla U\left(x\right)\right\Vert ,
\]
i.e. $\left(1,1\right)$-smoothness. The proposal scale $\sigma=\eta/d$
therefore gives a uniform positive acceptance probability for every
$0<\eta\leq\eta_{c}$, with $\eta_{c}$ as above. A more fine-grained
analysis shows that this is close to sharp; one can attain stable
acceptance probabilities for $\sigma\asymp\left(\log d\right)/d$,
but not for $\sigma\sim d^{-\left(1-\epsilon\right)}$ for any positive
$\epsilon$. 

This example lies far outside the setting of a globally continuous
gradient with a translation-invariant modulus, as $U$ and its derivatives
grow exponentially along rays. Nevertheless, the inverse proposal
scale required by the present argument remains only polynomial in
$d$. Since $U$ is $1$-strongly convex, standard isoperimetric estimates
can then be combined with the acceptance bound to yield polynomial
mixing guarantees. This exponential derivative growth is far from
fictional and occurs in various statistical models involving, e.g.,
Poisson likelihoods in their natural parameterization.

\subsection{Bounded-force, nonconvex perturbations}

The curvature-force condition is stable under bounded and smooth perturbations
of $\nabla U$. This provides a simple way to construct explicit nonconvex
examples.

Suppose that $V$ is $G$-smooth and that $W\in C^{2}\left(\mathbf{R}^{d}\right)$
satisfies
\[
\sup_{x}\left\Vert \nabla W\left(x\right)\right\Vert \leq B_{0},\qquad\sup_{x}\left\Vert \nabla^{2}W\left(x\right)\right\Vert _{\mathrm{op}}\leq B_{1}.
\]
For $U=V+W$, we have $\left\Vert \nabla V\left(x\right)\right\Vert \leq\left\Vert \nabla U\left(x\right)\right\Vert +B_{0}$
and hence 
\[
\begin{aligned}\left\Vert \nabla^{2}U\left(x\right)\right\Vert _{\mathrm{op}} & \leq\left\Vert \nabla^{2}V\left(x\right)\right\Vert _{\mathrm{op}}+\left\Vert \nabla^{2}W\left(x\right)\right\Vert _{\mathrm{op}}\\
 & \leq G\left(\left\Vert \nabla V\left(x\right)\right\Vert \right)+B_{1}\\
 & \leq G\left(\left\Vert \nabla U\left(x\right)\right\Vert +B_{0}\right)+B_{1}.
\end{aligned}
\]
Thus $U$ has curvature-force profile $\widetilde{G}\left(s\right)=B_{1}+G\left(s+B_{0}\right)$.
If $G$ is nondecreasing and concave, then so is $\widetilde{G}$.
In the affine case $G\left(s\right)=L_{0}+L_{1}\,s$, this reduces
to $\widetilde{G}\left(s\right)=\widetilde{L}_{0}+L_{1}\,s$ with
$\widetilde{L}_{0}=L_{0}+B_{1}+L_{1}\,B_{0}$.

For a concrete example, let $p>2$, let $v\in\mathbf{R}^{d}$ be a
unit vector, let $a>0$, and consider $U\left(x\right)=p^{-1}\left\Vert x\right\Vert ^{p}+a\,\cos\left(\langle v,x\rangle\right)$.
The oscillatory perturbation satisfies $\left\Vert \nabla W\left(x\right)\right\Vert \leq a$,
$\left\Vert \nabla^{2}W\left(x\right)\right\Vert _{\mathrm{op}}\leq a$.
Moreover, the Hessian of $U$ at the origin has eigenvalue $-a$ in
the direction $v$, so $U$ is nonconvex for every $a>0$.

With $\beta=(p-2)/(p-1)$, the preceding perturbation argument gives
that
\[
\left\Vert \nabla^{2}U\left(x\right)\right\Vert _{\mathrm{op}}\leq a+\left(p-1\right)\,\left(\left\Vert \nabla U\left(x\right)\right\Vert +a\right)^{\beta}.
\]
 Using the subadditivity of $s\mapsto s^{\beta}$, 
\[
\left\Vert \nabla^{2}U\left(x\right)\right\Vert _{\mathrm{op}}\leq\left(a+\left(p-1\right)\,a^{\beta}\right)+\left(p-1\right)\,\left\Vert \nabla U\left(x\right)\right\Vert ^{\beta}.
\]
 The nonconvex perturbation therefore has the same curvature-force
exponent as the unperturbed power potential. Set $L_{0,p,a}:=a+\left(p-1\right)\,a^{\beta}$,
$L_{\beta,p}:=p-1$. For sufficiently small $\eta_{p,c}>0$, take
\[
\sigma=\eta_{p,c}\,\min\left\{ \left(L_{0,p,a}\,d\right)^{-1/2},\left(L_{\beta,p}\,d\right)^{-(p-1)/p}\right\} .
\]
This gives a uniform positive acceptance probability. For fixed $p$
and $a$, this scale is again of order $d^{-(p-1)/p}$.

These examples also illustrate the separation between the local acceptance
analysis and the global geometry of the target. Curvature growth controls
the RWM acceptance probability, while convexity, isoperimetry, or
other assumptions needed to deduce mixing may be studied independently.
The contrast between the radial and product power examples also shows
where the general curvature-{}-force argument loses sharpness: it
controls the worst direction through the operator norm and therefore
cannot exploit the favourable effects of averaging across coordinates.

\section{Discussion}\label{sec:discussion}

The present results establish that steep potentials need not pose
a serious obstruction to the mixing of local, gradient-free Markov
chain Monte Carlo algorithms like the random-walk Metropolis. Some
natural refinements and extensions remain open.

First, while the curvature-{}-force profile is easy to work with and
supplies bounds which behave acceptably in all studied cases, it is
nevertheless far from sharp in some of these examples. This is due
in large part to the choice to summarise the local curvature of $U$
through the operator norm of its Hessian matrix, which obscures the
favourable averaging effects which are present in the definition of
$\alpha_{\sigma}\left(x\right)$. Some cursory calculations suggest
that imposing a similar gradient-domination assumption on the \emph{nuclear
norm} of $\nabla^{2}U\left(x\right)$ is likely to provide improved
dimension-dependence in these cases; this is left for future work.

Secondly, an exciting direction for future work would be to demonstrate
similar robustness for carefully designed first-order methods. Generalized
smoothness conditions such as $\left(L_{0},L_{1}\right)$-smoothness
first arose in the analysis of optimization methods motivated by gradient
clipping \citep{ZhangHSJ20}. It is natural to ask whether they might
also help analyse related sampling methods, including truncated, tamed
and implicit Langevin algorithms \citep{roberts1996exponential,atchade2006adaptive,brosse2019tamed,hodgkinson2021implicit},
the Barker proposal \citep{livingstone2022barker}, and other robust
MCMC strategies \citep{power2025some}. 

\section*{Declaration of the use of generative AI and AI-assisted technologies}

During the preparation of this work the author used OpenAI ChatGPT
and OpenAI Codex in order to identify useful smoothness assumptions,
check mathematical derivations, and assist with drafting and editing.
In particular, the use of curvature-{}-force profiles was identified
as the pivot for this analysis after ChatGPT suggested an assumption
which was more-or-less equivalent to the $\left(L_{0},L_{1}\right)$-smoothness
assumption of \citet{ZhangHSJ20}, which the author identified and
subsequently refined. After using these tools, the author reviewed
and edited the content as necessary and takes full responsibility
for the content of the publication.

\section*{Acknowledgement}

I thank Christophe Andrieu and Anthony Lee for productive discussions,
which motivated an expanded discussion of the sharpness of the results
herein, and a fuller discussion of the implications of these results
for mixing time estimates.

\bibliographystyle{plainnat}
\bibliography{steep_rwm}

\end{document}